\documentclass[11pt,a4paper]{article}\usepackage[latin1]{inputenc}  
\usepackage{amsmath}  
\usepackage{amssymb}  
\usepackage{curves}  
\usepackage{epic}  
\usepackage[all]{xy}   
\usepackage{graphicx}   
\usepackage{mathrsfs}
\usepackage{color}
\usepackage{euscript}
\usepackage{pdfsync}
\usepackage{enumerate}

\newtheorem{thm}{Theorem}[section]
\newtheorem{prop}[thm]{Proposition}

\newtheorem{lem}[thm]{Lemma}

\newtheorem{defn}[thm]{Definition}

\newtheorem{remark}[thm]{Remark}

\newtheorem{example}[thm]{Example}

\renewcommand{\thethm}{\thesection.\arabic{thm}}
\renewcommand{\thesection}{\arabic{section}}

\newenvironment{proof}
    {\par\noindent{\it Proof. }\nopagebreak\normalsize}%
    {\hfill\linebreak[2]\hspace*{\fill}$\square$\\ }

    {\par\noindent{\it Proof of Theorem \em \ref{main0}.\em}\nopagebreak\normalsize}%
    {\hfill\linebreak[2]\hspace*{\fill}$\square$\\ }

    {\par\noindent{\it Proof of Theorem \em \ref{main0}.\em}\nopagebreak\normalsize}%
    {\hfill\linebreak[2]\hspace*{\fill}$\square$\\ }

{\em}

      \newcommand{\N}{{\mathbb N}}
      \newcommand{\R}{{\mathbb R}}

\newcommand{\nmu}{{\overline{\mu}}}

\newcommand{\LIP}{\operatorname{LIP}}

\newcommand{\Mod}{\operatorname{Mod}}
\newcommand{\veps}{\varepsilon}
\newcommand{\jint}{\int\hspace{-4.15mm}\frac{\,\,\,}{}}
\newcommand{\jintb}{\int\hspace{-2.75mm}\frac{\,\,}{}}
\newcommand{\AM}{\operatorname{AM}}
\newcommand{\BV}{\operatorname{BV}}

\newcommand{\eps}{\varepsilon}
\newcommand{\pip}{\varphi}

\author{Estibalitz Durand-Cartagena, Sylvester Eriksson-Bique,\\ 
Riikka Korte, Nageswari Shanmugalingam\footnote{S.E.-B.~was partially supported by
grant DMS\#-1704215 of NSF(U.S.), R.K.~was supported by Academy of Finland grant number 308063, N.S.~was partially supported by
grant DMS\#-1500440 of NSF(U.S.A.), and E. D-C. 's research is partially supported by the grants 
MTM2015-65825-P (MINECO of Spain) and 2018-MAT14 (ETSI Industriales, UNED). Part of the research 
was done during the visit of the second and fourth authors to
Aalto University and Link\"oping University, and during the visit of the third and fourth authors to Universidad Complutense
de Madrid and UNED; the authors wish to thank these institutions for their kind hospitality.
We also thank
Olli Martio for many valuable discussions related to this subject and for sharing his early manuscript~\cite{HMM2}.}}

\begin{document}
\title{Equivalence of two BV classes of functions in metric spaces, and existence of a Semmes family of curves
under a $1$-Poincar\'e inequality}


\maketitle

\begin{abstract}
We consider two notions of functions of bounded variation in complete metric measure spaces,
one due to Martio~\cite{M1,M2} and the other due to Miranda~Jr.~\cite{Mi}. We show that these two notions 
coincide, if the measure is doubling and supports a $1$-Poincar\'e inequality. In doing so, we also prove that if the measure is doubling and supports a $1$-Poincar\'e inequality,
then the metric space supports a \emph{Semmes family of curves} structure. 
\end{abstract}



\section{Introduction}

Given $1\le p<\infty$, a function $u$ in $L^p(\R^n)$ is in $W^{1,p}(\R^n)$ if and only if $u$ has an 
$L^p$-representative that is absolutely continuous on almost every non-constant compact rectifiable
curve in $\R^n$ with derivative in $L^p(\R^n)$, see~\cite{Va} for an in-depth discussion on this.
Equivalently, $u\in W^{1,p}(\R^n)$ if and only if $u\in L^p(\R^n)$ and there is a non-negative Borel
function $g\in L^p(\R^n)$ such that for all non-constant compact rectifiable curves $\gamma$ in $\R^n$,
\begin{equation}\label{upgrad}
|u(\gamma(b))-u(\gamma(a))|\le \int_a^b g\circ\gamma(t)|\gamma^\prime(t)|\, dt.
\end{equation}
On the other hand, the class of BV functions on $\R^n$ has 
a more complicated analog; there should be
a sequence $f_k\in W^{1,1}(\R^n)$, with $
f_k\to u$ in $L^1(\R^n)$ and function $g_k$ 
associated with $f_k$ 
as in the inequality above, such that $\liminf_{k\to\infty} \int_{\R^n} g_k\, dx$ is finite. 
Thus to verify that
a function $u$ belongs to the class $\BV(\R^n)$ we need a sequence of 
\emph{pairs} of functions $(f_k,g_k)$ 
satisfying \eqref{upgrad}, where $f_k$ approximates $u$ in 
$L^1(\R^n)$, whereas to
define a function in $W^{1,1}(\R^n)$ we only need a single 
energy function $g$ that satisfies \eqref{upgrad}.

The above complication carries through from $\R^n$ to more general metric measure spaces $X$, and so
while we need only the energy function $g$ in order to know that $u$ is in the Sobolev class, to know that
$u$ is in the BV class we need both, the approximating sequence $f_k$ as well as the corresponding energy functions $g_k$. To avoid this 
discrepancy, the recent work of Martio~\cite{M1, M2} proposed a new definition of BV functions 
in the Euclidean and general metric measure setting,
denoted in the current paper by $\BV_{\AM}(X)$, see Definition \ref{def:BV-AM}.
In this notion one needs a single sequence of ``energy" functions $g_k$ associated
with the function $u$ in a specific manner in order to determine whether $u\in\BV_{\AM}(X)$.
The backbone
of the construction of $\BV_{\AM}(X)$ is the notion of $\AM$-modulus, and it appears that this modulus is
better suited to the study of sets of finite perimeter than the standard $1$-modulus. It is 
shown in~\cite[Theorem~11]{HMM2} that Euclidean Borel sets $E$ are of finite perimeter if and only if
the $\AM$-modulus of the collection of all curves that cross the measure-theoretic boundary $\partial_*E$ of
$E$ is finite; and in this case the perimeter measure of $E$ is precisely the $\AM$-modulus of that collection 
of curves. This is  a variant of the Federer characterization of sets of finite perimeter. Federer proved
that a measurable set $E\subset\R^n$ is of finite perimeter if and only if $\mathcal{H}^{n-1}(\partial_*E)$ is finite;
a new, potential-theoretic proof of this characterization, valid even in the metric setting, can be found in~\cite{L}.

The goal of this paper is to show that if the metric measure space $X$ is of controlled geometry, that is,
if $X$ is complete,
the measure $\mu$ is doubling and supports a $1$-Poincar\'e inequality, then the notion of
$\BV_{\AM}(X)$ from~\cite{M1,M2} gives the same function space as the BV class $\BV(X)$ as
defined by Miranda~Jr.~in~\cite{Mi}. To do so we also prove that if $\mu$ is doubling, then
$X$ supports a $1$-Poincar\'e inequality if and only if $X$ supports a \emph{Semmes family of curves}
corresponding to each pair $x,y\in X$ of points, that is, there is a family $\Gamma_{xy}$ of 
quasiconvex curves connecting $x$ to $y$ and a probability measure $\sigma_{xy}$ on
$\Gamma_{xy}$ satisfying a Riesz-type inequality, see Definition~\ref{def:semmes} below.
This auxiliary result is of independent interest. 
The notion of Semmes family of curves,
first proposed in~\cite{Se} (where clearly it was not termed a ``Semmes family"), is known to imply
the support of a $1$-Poincar\'e inequality, see the discussion in~\cite[page~29]{He}. In this paper we show
that the converse also holds true, that is, if the measure is doubling and supports a $1$-Poincar\'e inequality, then
it supports a Semmes family of curves structure. 
Thus, our paper also characterizes the support of a $1$-Poincar\'e inequality (in doubling complete metric measure
spaces) via the existence of a Semmes family of
curves.  A recent preprint~\cite{Fass-Orp} gives another characterization of the support of a $1$-Poincar\'e inequality in
terms of the existence of normal $1$-currents for each pair of points $x,y\in X$, in the sense of Ambrosio and Kirchcheim, such that the
mass of the current is controlled by the Riesz measure $R_{xy}$, see~\eqref{eq:Riesz} below. For the study comparing
BV-AM spaces with BV classes of functions, a Semmes family of curves seems to be more useful.

The equality of $\BV(X)$ with $\BV_{\AM}(X)$ and the
equivalence between the Semmes family of curves structure and the $1$-Poincar\'e inequality form
the two main results in this paper, see Theorem~\ref{thm:main1}.

\section{Two definitions of BV functions}

In the rest of the paper, $(X,d,\mu)$ is a 
{\em metric measure space}, where $(X,d)$ is a complete 
metric space and $\mu$ is a Borel measure. We denote by 
$B$ an open ball in $X$ and by $\lambda B$ the ball with 
the same center as $B$ and radius $\lambda$ times the 
radius of $B$.  Recall that the measure $\mu$ is said to be 
{\em doubling} if there is a constant $C\geq 1$ such that 
$\mu(2B)\leq  C\mu(B)$ for every ball $B$ in X. 

Given a compact interval $I\subset \R$, a {\em curve} $\gamma:I\to X$ is a continuous mapping. We only
consider curves that are non-constant and rectifiable. 
A curve 
$\gamma$, connecting two points $x,y\in X$, is {\em $C$-quasiconvex} if its length is at most $C\, d(x,y)$.

\subsection{$p$-Modulus and $\text{AM}$-modulus of a family of curves}

\begin{defn}\em\label{defmodinf}
Given a family $\Gamma$ of curves in $X$,  set $\mathcal{A}(\Gamma)$ 
to be the family of all
Borel measurable functions $\rho:X\rightarrow[0,\infty]$ such that
\[
\int_{\gamma}\rho \,  ds \geq 1\,\,\text{  for all  }\,\gamma\in\Gamma, 
\]
and set $\mathcal{A}_{seq}(\Gamma)$ to be the family of all
sequences $(\rho_i)$ of non-negative Borel measurable functions $\rho_i$ on $X$ such that
\[
\liminf_{i\to\infty}\int_{\gamma}\rho_i \,  ds \geq 1\,\,\text{  for all  }\,\gamma\in\Gamma.
\]
The integral $\int_\gamma\rho\, ds$ denotes the path-integral of $\gamma$ against the
arc-length re-parametrization of $\gamma$, see for example the description in~\cite{He}.
We define the $\infty$-\em modulus of $\Gamma$ \em by
\[
\Mod_{\infty}(\Gamma) = \inf_{\rho\in \mathcal{A}(\Gamma)}\, \|\rho\|_{L^{\infty}(X)},
\]
and for $1\le p<\infty$ the $p$-modulus of $\Gamma$ is
\[
 \Mod_p(\Gamma)=\inf_{\rho\in \mathcal{A}(\Gamma)}\, \int_X\rho^p\, d\mu.
 \]
Following~\cite{M1,M2}, we define 
 the \em approximate modulus \em ($\AM$-modulus) of $\Gamma$ by
\[
 \AM(\Gamma)=\inf_{(\rho_i)\in \mathcal{A}_{seq}(\Gamma)}\, \left\{\liminf_{i\to\infty}\int_X\rho_i\, d\mu\right\}.
 \]
The notion of $\AM_p(\Gamma)$ is defined analogously, with $\int_X\rho_i\, d\mu$ replaced by $\int_X\rho_i^p\, d\mu$.
If a property holds for all except for a family
$\Gamma$ of curves with $\Mod_p{\Gamma}=0$ (respectively with 
$\AM(\Gamma)=0$),
then we say that the property holds for \em $p$-a.e.~curve \em (respectively 
for \em $\AM$-a.e.~curve\em). 
\end{defn}

Note that $\AM(\Gamma)\le \Mod_1(\Gamma)$. Thanks to Mazur's lemma, it is a trivial consequence of the reflexivity of $L^{p}(X)$  that   
$\AM_p(\Gamma)=\Mod_p(\Gamma)$ when $1<p<\infty$, 
see \cite[Theorem 1]{HMM}. 
It is also easy to see that for any family of 
curves $\Gamma$ we  have $\AM_{\infty}(\Gamma)=
\Mod_{\infty}(\Gamma)$. Indeed, let $\tau=\AM_{\infty}
(\Gamma)$. If $\tau=\infty$ there is nothing to prove, so let us 
assume that $\tau<\infty$. By definition, there is a sequence of non-negative 
Borel functions $(g_i^{\veps})\in\mathcal{A}_{seq}(\Gamma)$ such that 
\[
\liminf_{i\to\infty}\|g_i^{\veps}\|_{L^{\infty}(X)}<\tau+\veps \quad
\text{and}\quad \liminf_{i\to\infty}\int_{\gamma}g_i^{\veps} \,  ds\geq 
1\text{ for each }\gamma\in\Gamma.
\]
Let $\rho_{\veps}:=\sup_i g_i^{\veps}$.  As $\rho_\veps\ge g_i^\veps$
for each $i\in\N$, it follows that 
\[
1\leq \liminf_{i\to\infty}\int_{\gamma}
g_i^{\veps} \,  ds\le  \int_{\gamma}\rho_{\veps} \,  ds,
\]
 and so 
$\Mod_{\infty}(\Gamma)\leq \|\rho_{\veps}\|_{L^{\infty}(X)}\leq \tau+\veps$ and 
the result follows.

Note that if every curve in $\Gamma$ is contained in a fixed ball $B$, then 
\[
\AM(\Gamma)\leq\Mod_1(\Gamma)\leq\, \mu(B)^{1-1/p}\Mod_p(\Gamma)^{1/p}\leq \mu(B)\, \Mod_\infty(\Gamma),
\]
and therefore
\[
  \limsup_{p\to\infty}\left[\Mod_p(\Gamma)\right]^{1/p}\leq\Mod_\infty(\Gamma).
\]


The next example shows that it is possible to have $\Mod_1(\Gamma)=\infty$ but 
$\AM(\Gamma)=1$. Further examples can be found in \cite[Section 9]{HMM}. 
The examples found there are families of curves that tangentially approach a smooth co-dimension one 
sub-manifold of $\R^n$.

\begin{example}\label{ex:Ex1}\em
Let $\Gamma$ be the collection of all rectifiable curves of length at most $1$ in the plane, 
and start from the $x$-axis 
with $0\le x\le 1$ and are parallel
to the $y$-axis. Then there is no acceptable $\rho\in L^1(X)$ for computing $\Mod_1(\Gamma)$, and hence 
$\Mod_1(\Gamma)=\infty$.
On the other hand, $\AM(\Gamma)$ is finite but positive. To see this, for each positive integer let 
$\rho_n=n\, \chi_{[0,1]\times[0,1/n]}$. Then $\int_\gamma\rho_n\, ds\ge 1$ whenever $\gamma$ is in $\Gamma$ with length at least
$1/n$, and as every curve in $\Gamma$ has positive length, we have that 
\[
\lim_{n\to\infty}\int_\gamma\rho_n\, ds\ge 1.
\]
So the sequence $(\rho_n)$ is admissible for $\Gamma$, and thus
\[
\AM(\Gamma)\le \limsup_{n\to\infty}\int_{\R^2}\rho_n\, d\mathcal{L}^2=\limsup_{n\to\infty}n\, \left( \frac{1}{n}\times 1\right)=1.
\]
To see that $\AM(\Gamma)>0$, we consider the sub-family $\Gamma_{1/2}$ of all line segments in $\Gamma$ with length $1/2$,
and let $(\rho_i)\in \mathcal{A}_{seq}(\Gamma_{1/2})$. Then by Fubini's theorem, for each $i\in\N$ we have
\[
\int_{\R^2}\rho_i\, d\mathcal{L}^2\ge \int_0^1\int_0^{1/2}\rho_i(x,y)\, dy\, dx
= \int_0^1\left(\int_0^{1/2}\rho_i(x,y)\, dy\right)\, dx. 
\]
Now by Fatou's lemma, 
\[
\liminf_{i\to\infty}\int_{\R^2}\rho_i\, d\mathcal{L}^2\ge\int_0^1\left(\liminf_{i\to\infty}\int_0^{1/2}\rho_i(x,y)\, dy\right)\, dx\ge 1.
\]
It follows that 
\[
\AM(\Gamma)\ge \AM(\Gamma_{1/2})\ge 1. 
\]
\end{example}



\subsection{BV functions based on the notion of $\AM$-modulus.}

\begin{defn}\em
A nonnegative Borel function $g$ on $X$ is a \em $1$-weak upper
gradient \em of an extended real-valued function $u$ on X if for $1$-a.e. curve $\gamma : [a,b] \to X$,
$$
|u(\gamma(a))-u(\gamma(b))|\leq\int_{\gamma}g\, ds.
$$
\end{defn}

Given a function $u$ that has a $1$-weak upper gradient in $L^1(X)$, there is a \emph{minimal}
$1$-weak upper gradient of $u$, denoted $g_u$, in the sense that whenever $g$ is a $1$-weak upper 
gradient of $u$, we have $g_u\le g$ almost everywhere in $X$.

The following notion of BV functions on $X$ is due to Miranda Jr.~\cite{Mi}.

\begin{defn}\em {\rm {\bf (BV functions)}}
For $u\in L^1_{\text{loc}}(X)$, we define the total variation of $u$ as
$$
\|Du\|(X):=\inf\left\{\liminf_{i\to\infty}\inf_{g_{u_i}}\int_X g_{u_i} d\mu:u_i\in\LIP_{\text{loc}}(X),u_i\to u \,\text{in}\,L^1_{\text{loc}}(X)\right\},
$$
where the second infimum is over all $1$-weak upper gradients $g_{u_i}$ of $u_i$. 
We say that a function $u\in L^1_{\text{loc}}(X)$ is of 
{\em bounded variation}, $u\in \BV(X)$ if $\|Du\|(X)<\infty$. A measurable set $E\subset X$ is said of 
{\em finite perimeter} if  $\|D\chi_E\|(X)<\infty$.
\end{defn}

 The following definition of $\BV_{\AM}$ class is from~\cite{M1}. 

\begin{defn} {\rm {\bf (BV-AM functions)}}\label{def:BV-AM}
\em  A function $u\in L^1(X)$ is in the $\BV_{\AM}(X)$ class if there is 
a family $\Gamma$ of rectifiable curves in $X$ with $\AM(\Gamma)=0$, and
a sequence $(g_i)$
of non-negative Borel measurable functions in $L^1(X)$ such that whenever $\gamma:[a,b]\to X$ is a
non-constant compact rectifiable curve that does not belong to $\Gamma$, 
we have that
\begin{equation}\label{eq:AM-uppr}
|u(\gamma(t))-u(\gamma(s))|\le \liminf_{i\to\infty}\int_{\gamma\vert_{[s,t]}} g_i\, ds
\end{equation}
for $\mathcal{H}^1$-a.e.~$s,t\in[a,b]$ with $s<t$,
and 
\[
\liminf_{i\to\infty}\int_X g_i\, d\mu<\infty.
\]
Such a sequence $(g_i)$ is said to be a \em $\BV_{\AM}$-upper bound \em of $u$.
We set
\[
\Vert D_{\AM}u\Vert(X):=\inf_{(g_i)}\liminf_{i\to\infty}\int_X g_i\, d\mu,
\]
where the infimum is over all $\BV_{\AM}$-upper bounds of $u$.
\end{defn}

Notice that by \cite[Theorem 4.1]{M2}, $\BV(X)\subseteq\BV_{\AM}(X)$. This also follows from the next lemma.
The following lemma holds even if $\mu$ is not doubling or does not support a $1$-Poincar\'e inequality.

\begin{lem}\label{lem:BV2BVN}
Assume that $u \in \BV(X)$. Then
there is a set $N\subset X$ with $\mu(N)=0$ and
a sequence $(g_i)$
of non-negative Borel measurable functions in $L^1(X)$ such that whenever $\gamma$ is a
non-constant compact rectifiable curve 
with end-points $x,y\in X\setminus N$,
\begin{equation*} 
|u(y)-u(x)|\le \liminf_{i\to\infty}\int_\gamma g_i\, ds
\end{equation*}
{\em(}that is,~\eqref{eq:AM-uppr} holds{\em)} and 
\[
\liminf_{i\to\infty}\int_X g_i\, d\mu<\infty.
\]
\end{lem}

Note that the lemma gives a stronger control of $u$ than allowed by the $\BV_{\AM}$-control. For
functions in $\BV_{\AM}(X)$, we know that given a path $\gamma$ there is a set $N_\gamma$ with 
$\mathcal{H}^1(\gamma^{-1}(N_\gamma))=0$ so that whenever $x,y$ lie in the trajectory of $\gamma$ with
$x,y\not\in N_\gamma$, inequality~\eqref{eq:AM-uppr} holds. Here we show that we can choose $N_\gamma$ to
be independent of $\gamma$ and in addition with 
$\mu$-measure zero.

\begin{proof} 
%
Given $u\in\BV(X)$ there is a sequence
$u_i \in\LIP_{\text{loc}}(X)$
such that $u_i\to u$ in $L^1(X)$ and $\lim_{i\to\infty}\int_Xg_{i}\, d\mu\le M<\infty$ for a choice
of upper gradients $g_i$ of $u_i$. By passing to a 
subsequence if necessary, we may also assume that $u_i\to u$ pointwise $\mu$-a.e.~in $X$. Let $N$ be the set of all
points $x\in X$ for which $\lim_{i\to\infty}u_i(x)\ne u(x)$. Then $\mu(N)=0$. Let $\gamma$ be a non-constant compact
rectifiable curve in $X$ with end points $x,y\in X\setminus N$. Then
\[
|u(x)-u(y)|=\lim_{i\to\infty}|u_i(x)-u_i(y)|\le \liminf_{i\to\infty}\int_\gamma g_i\, ds.
\]
\end{proof}

The main focus of this paper is to show that $\BV_{\AM}(X)=\BV(X)$ when the measure on $X$ is doubling
and supports a $1$-Poincar\'e inequality. 

%

\subsection{The spaces $N^{1,1}(X)$ and $N^{1,1}_{\AM}(X)$}

Let $\widetilde{N}^{1,1}(X,d,\mu)$, where $1\leq p<\infty$, be the
class of all $L^1$-integrable Borel functions on $X$ for which there
exists a $1$-weak upper gradient in $L^1(X)$. For
$u\in\widetilde{N}^{1,1}(X,d,\mu)$ we define
$$
\|u\|_{\widetilde{N}^{1,1}(X)}=\|u\|_{L^1(X)}+\inf_{g}\|g\|_{L^1(X)},
$$
where the infimum is taken over all $1$-weak upper gradients $g$ of
$u$. As usual, we can now define $N^{1,1}(X,d,\mu)$ equipped with the norm
$\|u\|_{N^{1,1}(X)}=\|u\|_{\widetilde{N}^{1,1}(X)}.$

Once we have the new concept of $\AM$-a.e.~curve, it is natural to define an upper gradient 
and a Sobolev class related to this notion.

\begin{defn}[Weak AM-upper gradient]\em
A nonnegative Borel function $g$ on $X$ is a \em weak $\AM$-upper gradient \em of $u$ on X if
$|u(\gamma(a))-u(\gamma(b))|\leq\int_{\gamma}g\, ds$
for $\AM$-a.e. curve $\gamma:[a,b]\to X$.
\end{defn}

\begin{defn}\em {\rm {\bf ($N^{1,1}_{\AM}$ functions)}}
Let $\widetilde{N}^{1,1}_{\AM}(X,d,\mu)$, be the
class of all Borel functions in $L^1(X)$ for which there
exists a weak $\AM$-upper gradient in $L^1(X)$. 
For
$u\in\widetilde{N}^{1,1}_{\AM}(X)$ we define
$$
\|u\|_{\widetilde{N}^{1,1}_{\AM}(X)}=\|u\|_{L^1(X)}+\inf_{g}\|g\|_{L^1(X)},
$$
where the infimum is taken over all weak $\AM$-upper gradient $g$ of
$u$. We can now define $N^{1,1}_{\AM}(X)$ 
to be the class $\widetilde{N}^{1,1}_{\AM}(X,d,\mu)$,
equipped with the norm
$\|u\|_{N^{1,1}_{\AM}(X)}=\|u\|_{\widetilde{N}^{1,1}_{\AM}(X)}$.
\end{defn}



The following lemma proves that the first definition implies the second one. In some sense, the first 
definition is related to the Sobolev class $N^{1,1}$ while the second is related to the $\BV$ class.

\begin{lem}\label{lem:AMweak-AMupper}
If a function $u$ on X has  $g$ as a weak $\AM$-upper gradient, then there exists a $\BV_{\AM}$-upper bound of $u$.
\end{lem}

\begin{proof} 
Assume that 
\begin{equation}\label{eq:BV-upp}
|u(\gamma(a))-u(\gamma(b))|\leq\int_{\gamma}g \, ds
\end{equation}
for $\AM$-a.e. curve $\gamma:[a,b]\to X$. Let $\Gamma$ be the 
collection of curves for which~\eqref{eq:BV-upp} does not hold. By 
definition $\AM(\Gamma)=0$ and so by ~\cite[Theorem~7]
{HMM} there is a sequence of non-negative Borel functions 
$\widetilde{g}_i$ such that 
\[
\liminf_{i\to\infty}\|\widetilde{g}_i\|_{L^1}<\infty
\quad\text{and }\quad  \liminf_{i\to\infty}\int_{\gamma}\widetilde{g}_i \,  ds=\infty\,\, \text{ for all }\gamma\in\Gamma.
\]
Let $\Gamma_0$ be the collection of all non-constant compact rectifiable curves $\gamma$ in $X$ for which
\[
 \liminf_{i\to\infty}\int_{\gamma}\widetilde{g}_i \,  ds=\infty;
\]
then $\AM(\Gamma_0)=0$. Observe that if $\gamma$ is a non-constant compact rectifiable curve in $X$
such that $\gamma\not\in\Gamma_0$, then every sub-curve of $\gamma$ also does not belong to $\Gamma_0$.
Now, for each $\eps>0$ the sequence of functions $g_i=g+\eps\widetilde{g}_i$ has the property 
 that for $\gamma\notin\Gamma_0$,
 \[
 |u(\gamma(a))-u(\gamma(b))|\leq\int_\gamma g\, ds\leq\liminf_{i\to\infty}\int_{\gamma}(g+\eps\widetilde{g}_i)\, ds,
\]
and for $\gamma\in\Gamma_0$,
\[
 |u(\gamma(a))-u(\gamma(b))|\leq\infty=\int_{\gamma}(g+\eps\widetilde{g}_i)\, ds,
\]
so $|u(\gamma(a))-u(\gamma(b))|\leq\liminf_{i\to\infty}\int_{\gamma}g_i\, ds
$ holds for every curve $\gamma$.
\end{proof}

Note that we have more than just that the sequence $(g_i)$ forms  a $\BV_{\AM}$-upper bound 
of $u$; the inequality holds for \emph{every} subcurve of $\gamma$, not merely for $\mathcal{H}^1$-almost every
pair of points in the domain of $\gamma$.

From the above we know that for $1<p<\infty$, $N^{1,1}(X)\subseteq N^{1,1}_{\AM}(X)\subsetneq\BV_{\AM}(X)$ and
$$
\LIP^{\infty}(X)\subseteq N^{1,\infty}(X) \subseteq N^{1,p}(X)\subseteq N^{1,1}(X)\subsetneq\BV(X)
 \subseteq \BV_{\AM}(X).
$$
In Section~\ref{sec:PI-BV} we will show that if $X$ supports a $1$-Poincar\'e inequality then 
$\BV_{\AM}(X)=\BV(X)$ and that $N^{1,1}_{\AM}(X)=N^{1,1}(X)$.



\begin{remark}\label{rem:truncation} \em
For $u\in \BV_{\AM}(X)$ and a sequence $(g_i)$ such that
$\lim_{i\to\infty}\int_X g_i\, d\mu<\infty$, the sequence of measures $(g_i\, d\mu)$ is a bounded
sequence. We can assume (by localizing the argument if need be) that $X$ is compact as well. Then there is
a subsequence, also denoted $(g_i\, d\mu)$,
and a Radon measure $\nu$ on $X$ such that the sequence of measures
$(g_i\, d\mu)$ converges weakly* to $d\nu$ in $X$. As $X$ is compact, we see that
$\Vert D_{\AM}u\Vert(X)\le \nu(X)$. 
\end{remark}

\section{Equivalence of BV and AM-BV classes under Poincar\'e inequality}\label{sec:PI-BV}

The aim of this section is to show the 
equivalence of the functional spaces $\BV(X)$ and $
\BV_{\AM}(X)$, under the additional hypothesis that the 
metric space supports a doubling measure and a $1$-Poincar\'e inequality. 

\begin{defn}\em
The metric measure space $X$ supports a {\em $1$-Poincar\'e inequality} if there are positive constants $C,\lambda$
such that whenever $B$ is a ball in $X$ and $g$ is 
an upper gradient of $u$,
\[
 \jint_B|u-u_B|\, d\mu\le C\, \text{rad}(B)\, \jint_{\lambda B}g\, d\mu.
\]
Here $u_B:=\mu(B)^{-1}\int_B u\, d\mu=\jintb_B u\, d\mu$ is 
the average of $u$ on the ball $B$.
\end{defn}

With the notion of $\BV_{\AM}$ class, one could even define a stronger version of $1$-Poincar\'e inequality. 

\begin{defn}\em
We say that $X$ supports an \em $\AM$-Poincar\'e 
inequality \em if there exist constants
$C>0$, $\lambda\ge 1$
such that for each measurable function $u$ on $X$, each $\BV$-upper bound $(g_i)$ of $u$, 
and each ball $B\subset X$, we have
\[
 \jint_B|u-u_B|\, d\mu\le C\, \text{rad}(B)\, \liminf_{i\to\infty}\jint_{\lambda B} g_i\, d\mu.
\]
\end{defn}

This should imply that 
\[
 \jint_B|u-u_B|\, d\mu\le C\, \text{rad}(B)\, \frac{\Vert D_{\AM}u\Vert(\lambda B)}{\mu(\lambda B)}.
\]

On the other hand, notice that $1$-Poincar\'e inequality  implies 
\[
 \jint_B|u-u_B|\, d\mu\le C\, \text{rad}(B)\frac{\|Du\|(\tau B)}{\mu(\tau B)}.
 \]

 As a first step, in the following proposition we prove the equivalence of $\BV(X)$ and 
$\BV_{\AM}(X)$ under the hypotheses that the measure is doubling and the space supports an 
$\AM$-Poincar\'e 
inequality. We will see in Theorem~\ref{thm:main1} that the support of an $\AM$-Poincar\'e 
inequality is equivalent to the support of a $1$-Poincar\'e  inequality. 

\begin{prop}\label{prop:BV-AM-PI}
If $X$ supports a $\AM$-Poincar\'e 
inequality and $\mu$ is doubling, then the two classes $\BV_{\AM}(X)$ and
$\BV(X)$ are equal, with comparable norms. 
\end{prop}

\begin{proof} 
 Note first that $\BV(X)\subset \BV_{\AM}(X)$, see Lemma~\ref{lem:BV2BVN}.
  
 Now 
 let us prove that if $u\in \BV_{\AM}(X)$, then $u\in \BV(X)$.  By the doubling property of $\mu$, for $\eps>0$
we can cover $X$ by balls $B_i=B(x_i,\eps)$ such that the balls $5\lambda B_i$ have bounded overlap. Let 
$\pip_i^\eps$ be a partition of unity subordinate to the cover $2B_i$. For $u\in \BV_{\AM}(X)$ let
\[
u_\eps=\sum_i u_{B_i}\pip_i^\eps.
\]
Recall that we have bounded overlap of the collection $5B_i$ with
$X=\bigcup_j B_j$,
$\mu$ is doubling, and that if $2B_i$ intersects $B_j$ then
$5B_j\supset 2B_i$.
Then we have for $x\in B_j\subset X$, 
\begin{align*}
|u(x)-u_\eps(x)|=\bigg\vert\sum_i [u_{B_i}-u(x)]\pip_i^\eps(x)\bigg\vert
   &\le \sum_i|u_{B_i}-u(x)|\pip_i^\eps(x)\\
   &\le \sum_{i\, :\, x\in 2B_i}|u_{B_i}-u(x)|\\
   &\le \sum_{i\, :\, 2B_i\cap B_j\ne\emptyset}|u_{B_i}-u(x)|\\
   &\le C\, C_D^3\, |u_{5B_j}-u(x)|.
\end{align*}
Therefore, by the AM-Poincar\'e inequality,
\begin{align*}
\int_X|u-u_\eps|\, d\mu \le \sum_j\int_{B_j}|u-u_\eps|\, d\mu
 &\le C\sum_j \int_{B_j}|u_{5B_j}-u|\, d\mu\\
 &\le C\sum_j\int_{5B_j}|u-u_{5B_j}|\, d\mu\\
 &\le C\eps \sum_j\Vert D_{\AM}u\Vert(5\lambda B_j).
\end{align*}
Since $\Vert D_{\AM}u\Vert$ is a Radon measure (\cite[Theorem 3.4]{M1}) and $5\lambda B_j$ have bounded
overlap, we have
\[
\int_X|u-u_\eps|\, d\mu\le C\eps\, \Vert D_{\AM}u\Vert(X)\to 0\text{ as }\eps\to 0^+.
\]
Thus $u_\eps\to u$ in $L^1(X)$, and we also know from the definition of $u_\eps$ that each $u_\eps$ is locally
Lipschitz and hence in $N^{1,1}_{loc}(X)$.
Next, for $x,y\in B_j$,
\begin{align*}
|u_\eps(x)-u_\eps(y)|=\bigg\vert\sum_iu_{B_i}[\pip_i^\eps(x)-\pip_i^\eps(y)]\bigg\vert
  &=\bigg\vert\sum_i[u_{B_i}-u_{B_j}][\pip_i^\eps(x)-\pip_i^\eps(y)]\bigg\vert\\
  &\le \frac{d(x,y)}{\eps}\sum_{i\, :\, 2B_i\cap B_j\ne\emptyset}|u_{B_i}-u_{B_j}|\\
  &\le \frac{C\, d(x,y)}{\eps}\sum_{i\, :\, 2B_i\cap B_j\ne\emptyset}\jint_{5B_j}|u-u_{5B_j}|\, d\mu.
\end{align*}
It follows from the bounded overlap of $5B_i$ that
\[
\text{Lip}u_\eps(x)\le \frac{C}{\eps}\jint_{5B_j}|u-u_{5B_j}|\, d\mu
\]
whenever $x\in B_j$. 
Integrating  the above inequality over $X=\bigcup_jB_j$, we obtain
\begin{align}
\int_X\text{Lip}u_\eps\, d\mu &\le \frac{C}{\eps}\sum_j\mu(B_j)\, \jint_{5B_j}|u-u_{5B_j}|\, d\mu\notag\\
  &\le \frac{C}{\eps}\sum_j\int_{5B_j}|u-u_{5B_j}|\, d\mu\notag\\
  &\le \frac{C}{\eps}\sum_j \eps\, \Vert D_{\AM}u\Vert(5\lambda B_j)\notag\\
  &\le C\, \Vert D_{\AM}u\Vert(X).\label{eq:approx-Lip}
\end{align}
Thus 
\[
  \liminf_{\eps\to 0^+}\int_Xg_{u_\eps}\, d\mu\le C\, \Vert D_{\AM}u\Vert(X)<\infty,
\]
and as $u_\eps\to u$ in $L^1(X)$, it follows that $u\in \BV(X)$ with
$\Vert Du\Vert(X)\le C\, \Vert D_{\AM}u\Vert(X)$. 

We also have $\Vert D_{\AM}u\Vert(X)\le \Vert Du\Vert(X)$, as we now show.
Suppose now that $\Vert Du\Vert(X)$ is finite, and let $u_k\in \BV(X)$ be such that $u_k\to u$ in $L^1(X)$ and 
$\lim_{k\to\infty}\int_Xg_{u_k}\, d\mu=\Vert Du\Vert(X)$. By passing to a subsequence if necessary, we may also assume that
$u_k\to u$ pointwise almost everywhere in $X$ as well. 
For each $k\in\N$ we choose an upper gradient $g_k$ of $u_k$ such that $\int_Xg_k\, d\mu\le \int_X g_{u_k}\, d\mu+\eps/2^k$.
We set $N$ to be the collection of all points
$x\in X$ at which $u_k(x)$ does not converge to $u(x)$. Then $\mu(N)=0$, and so the $1$-modulus of the collection 
$\widehat{\Gamma}_N^+$ of non-constant compact rectifiable curves $\gamma$ in $X$ for which $\mathcal{H}(\gamma^{-1}(N))>0$
is zero. Using 
\cite[(7.8)]{He},
we know that
the collection $\Gamma^+_N$ of all non-constant compact rectifiable curves in $X$ with a subcurve in $\widehat{\Gamma}_N^+$
is also of $1$-modulus zero. Let $\gamma$ be a non-constant compact rectifiable curve in $X$ with $\gamma\not\in\Gamma_N^+$.
By re-parametrizing if necessary, we now assume that $\gamma:[a,b]\to X$ is arc-length parametrized; 
then $\mathcal{H}^1([a,b]\setminus\gamma^{-1}(N))=0$. For $s,t\in[a,b]\setminus\gamma^{-1}(N)$ with
$s>t$ we have that
\[
|u(\gamma(t))-u(\gamma(s))|=\lim_{k\to\infty}|u_k(\gamma(t))-u_k(\gamma(s))|
  \le \liminf_{k\to\infty}\int_{\gamma\vert_{[t,s]}}\, g_k\, ds\le \liminf_{k\to\infty}\int_\gamma g_k\, ds.
\]
This verifies that  $(g_k)$ is a $\BV_{\AM}$-upper bound for $u$ in the sense of 
Definition~\ref{def:BV-AM}.
\end{proof}

\begin{prop}\label{prop:N1p-AM-PI}
If $X$ supports an $\AM$-Poincar\'e 
inequality and $\mu$ is doubling, then $N^{1,1}_{\AM}(X)=N^{1,1}(X)$ with comparable
norms. 
\end{prop} 

\begin{proof}
Note that $N^{1,1}(X)\subset N^{1,1}_{\AM}(X)$.
Thus it suffices to prove the reverse inclusion. Let $u\in N^{1,1}_{\AM}(X)$, and let $g\in L^1(X)
$ be a weak AM-upper gradient
of $u$. Let $\Gamma$ be the corresponding exceptional family; then $\AM(\Gamma)=0$. Then 
by the proof of
Lemma~\ref{lem:AMweak-AMupper} there is a sequence $(\rho_i)$ of non-negative Borel functions in $L^1(X)$
such that $\int_X\rho_i\, d\mu\le M<\infty$ for each $i\in\N$ and for each $\gamma\in\Gamma$ 
we have
\[
\lim_{i\to\infty}\int_\gamma\rho_i\, ds=\infty.
\]
Then for each $\eps>0$ we have that $(g+\eps\rho_i)$
forms a $\BV_{\AM}$-upper bound of $u$, and so
as $X$ supports an $\AM$-Poincar\'e inequality,
whenever $B$ is a ball in $X$ we have
\[
\jint_B|u-u_B|\, d\mu\le \frac{C\, \text{rad}(B)}{\mu(B)}\, \liminf_{i\to\infty}\int_{\lambda B}[g+\eps\rho_i]\, d\mu.
\]
As before, by passing to a subsequence if necessary, we may assume that $\rho_i\, d\mu$ 
converges weakly to
a Radon measure $\nu$ on $X$, and so the above turns into 
\[
\jint_B|u-u_B|\, d\mu\le C\, \text{rad}(B)\, \left(\jint_{\lambda B}g\, d\mu+\eps\frac{\nu(2\lambda B)}{\mu(2\lambda B)}\right).
\]
Letting $\eps\to 0$ we get
\[
\jint_B|u-u_B|\, d\mu\le C\, \text{rad}(B)\, \jint_{\lambda B}g\, d\mu.
\]
We now know from Proposition~\ref{prop:BV-AM-PI} that $u\in \BV(X)$.
Now an argument as in the proof of Proposition~\ref{prop:BV-AM-PI},
up to and including~\eqref{eq:approx-Lip}, applied to open sets $U\subset X$ with $\mu(\partial U)=0$, we obtain that
\[
\Vert Du\Vert(U)\le C\, \int_{\overline{U}} g\, d\mu=\int_U g\, d\mu.
\]
Note that $g\in L^1(X)$, and hence for each $\eta>0$ there is some $\eps>0$ such that 
whenever $K\subset X$ is measurable with $\mu(K)<\eps$, we have $\int_Kg\, d\mu<\eta$.
Since whenever $E\subset X$ with $\mu(E)=0$, for each $\eps>0$ we can find an open set $U_\eps\supset E$
such that $\mu(U_\eps)<\eps$ and $\mu(\partial U_\eps)=0$,
it follows that $\Vert Du\Vert\ll \mu$, and hence $u\in N^{1,1}(X)$ by~\cite[Theorem~4.6]{HKLL}.
%
\end{proof}

Note that if $X$ does not support a $1$-Poincar\'e inequality, 
we do not know the equivalence of
$N^{1,1}(X)$ with $N^{1,1}_{\AM}(X)$. Similar difficulties show up in comparing other alternative
notions of $N^{1,1}(X)$ as well, see for example~\cite[Section~8]{ADiM}. 
We will prove in Theorem~\ref{thm:main1} that $X$ supports a $1$-Poincar\'e inequality if and only if it 
supports the a priori stronger $\AM$-Poincar\'e inequality.

The key point in the above proof is that if $u\in \BV(X)$ and $\Vert Du\Vert\ll\mu$, then
$u\in N^{1,1}(X)$; the validity of this point requires a doubling measure 
supporting a $1$-Poincar\'e inequality.
The following counterexample
is from~\cite[Example~7.4]{ADiM}. 
We do not have a counterexample for the statement 
``$\Vert Du\Vert\ll\mu$ implies $u\in N^{1,1}(X)$" in the case $\mu$ is doubling,
but the measure $\mu$ in the following example is asymptotically doubling.

\begin{example}\em
Let $X=\R^2$ be equipped with the Euclidean metric and the measure $\mu=\mathcal{L}^2+\mathcal{H}^1\vert_C$
where $C$ is the boundary of the unit disk $D$ in $\R^2$ centered at the origin. Let $u=\chi_D$. Then,
by the approximations $f_\eps(x)=(1-\eps^{-1}\text{dist}(x,D))_+$ of $u$ we see that 
$u\in \BV(X)$ with $\Vert Du\Vert\equiv\mathcal{H}^1\vert_C$. It follows that $\Vert Du\Vert\ll\mu$.
However, $u\not\in N^{1,1}(X)$: with $\Gamma$ the collection of all line segments $\gamma_x$, $-1<x<1$,
given by $\gamma_x:[-2,2]\to X$ where $\gamma_x(t)=(x,t)$, we have that $u\circ\gamma_x$ is not
absolutely continuous on $[-2,2]$, and furthermore, $\Mod_1(\Gamma)>0$. 
\end{example}

 The existence of a Semmes family of curves provides a key tool for the proof that the 
AM-Poincar\'e inequality and the standard $1$-Poincar\'e inequality are equivalent, which in turn 
allows us to prove equivalence of 
the two classes $\BV(X)$ and $\BV_{\AM}(X)$ with just the assumption of a $1$-Poincar\'e inequality
in addition to the doubling property of $\mu$. 
Thus we next prove that the existence of $1$-Poincar\'e
inequality in the doubling complete metric measure space $X$ is equivalent to the existence of the following 
Semmes pencil of curves. See~\cite{Fass-Orp} for a closely related characterization of the $1$-Poincar\'e 
inequality in terms of $1$-currents in the sense of Ambrosio and Kirchheim~\cite{AK}.

 If $A$ is a Borel subset of $X$ and $\gamma$ is a rectifiable curve, we define 
$\ell(\gamma \cap A):=\mathcal{H}^1(\gamma\cap A)$.

\begin{defn} \label{def:semmes}\em (\cite{Se, He})
A space $X$ supports a \em Semmes pencil of curves \em if there exists a constant $C>0$ such that for 
each pair of points $x,y\in X$ with $x\neq y$ there is a family $\Gamma_{xy}$ of rectifiable curves in $X$ 
equipped with a probability measure $d\sigma=d\sigma_{x,y}$ so that each $\gamma\in \Gamma_{xy}$ is 
a $C$-quasiconvex curve joining $x$ to $y$, and for each Borel set $A\subset X$, the map 
$\gamma\mapsto\ell(\gamma\cap A)$ is $\sigma$-measurable and satisfies
\begin{equation}\label{pencil}
\int_{\Gamma_{xy}}\ell(\gamma\cap A)\, d\sigma(\gamma)\le C\int_{C B_{x,y}\cap A}R_{xy}(z)\, d\mu(z).
\end{equation}
In the previous inequality, for $C>0$, $C B_{x,y}:=B(x,Cd(x,y))\cup B(y,Cd(x,y))$ and
\begin{equation}\label{eq:Riesz}
R_{xy}(z) := \frac{d(x,z)}{\mu(B(x,d(x,z)))}+\frac{d(y,z)}{\mu(B(y,d(y,z)))}.
\end{equation}
\end{defn}

We next show that if the measure on $X$ is doubling and supports a $1$-Poincar\'e inequality, then it supports a Semmes
pencil of curves.

Denote 
\begin{equation}
\Gamma^C_{xy} := \{(\gamma,I)\, :\, \text{ curve }\gamma:I\to X \text{ is $1$-Lipschitz, } \gamma(0)=x, \gamma(\max(I))=y\},
\end{equation}
where $I$ are intervals contained in $[0,C\, d(x,y)]$ with left-hand end point $0$. 
We equip $\Gamma^C_{xy}$ with the following metric.
The elements of $\Gamma^C_{xy}$ can be identified with their graphs 
\[
\Gamma_\gamma = \{(t,\gamma(t))\, :\, t \in I\} \subset \R \times X. 
\]
We define a metric on
$\Gamma^C_{xy}$ by setting 
\[
d(\gamma,\gamma') := d_H(\Gamma_\gamma, \Gamma_{\gamma'}), 
\]
where $d_H$ is the Hausdorff metric. Thanks to the Arzela-Ascoli theorem, 
this metric makes $\Gamma^C_{xy}$ into a complete and compact metric space
because $X$ is complete and doubling and hence closed bounded subsets of $X$ are compact.
For $f \in C(X)$, the functional $\Phi_f : \Gamma^C_{xy} \to \R$ given by
\[
\Phi_f((\gamma,I)) := \int_I f \circ \gamma \,dt,
\]
is continuous on $\Gamma^C_{xy}$. 

We denote the Riesz measure by
\[
d\nmu^C_{xy}(z) = R_{xy} ~d\mu|_{CB_{x,y}}.
\]

\begin{thm} \label{thm:semmesexist} 
If $(X,d,\mu)$ satisfies a $1$-Poincar\'e inequality, then there exists  
$C \geq 1$ such that for every $x,y \in X$ with $x\ne y$, there exist a compact 
family of curves $\Gamma_{xy}$ and a Radon probability measure $\alpha_{xy}$ on 
$\Gamma_{xy}$ which constitutes a Semmes family of curves, i.e. for every Borel set $A$,
\[
\int_{\Gamma^C_{xy}} \int_\gamma \chi_A\, ds\, d\alpha_{xy}(\gamma) 
\leq C \int_{CB_{x,y} \cap A} R_{xy}(z) \, d\mu(z) = \nmu^C_{xy}(A).
\]
\end{thm}

The proof of the above statement could be derived by a careful application of the techniques in \cite{AMS} combined 
with the modulus estimates of~\cite{K}. However, the method in~\cite{AMS} directly works only for $p>1$, 
and some additional care is necessary for $p=1$. Further, the following proof is somewhat 
more direct than theirs. Our proof is more in line with the approaches in~\cite{B,S} 
in combination with the estimates from~\cite{K}
to construct probability measures on the space of curve fragments. The papers~\cite{B,S} employ the 
Rainwater lemma from~\cite[Theorem 9.4.3]{R2}. However, we are able to avoid this lemma 
by directly using the Min-Max theorem~\cite[Theorem 9.4]{R2}, restated below for the reader's
convenience. 

\begin{prop}{{\em (Min-Max Theorem~\cite[Theorem~9.4.2]{R2})}} \label{thm:minmax} 
Suppose that
\begin{enumerate}[(i)]
\item $G$ is a convex subset of some vector space,
\item $K$ is a compact convex subset of some topological vector space, and
\item $F : G \times K \to \R$ satisfies
\begin{enumerate}[(a)]
\item $F(\cdot, y)$ is convex on $G$ for every $y \in K$,
\item $F(x,\cdot)$ is concave and continuous on $K$ for every $x \in G$.
\end{enumerate}
\end{enumerate}
Then
\[
\sup_{y \in K} \inf_{x \in G} F(x,y) = \inf_{x \in G} \sup_{y \in K} F(x,y).
\]
\end{prop}

\begin{proof}[Proof of Theorem~\ref{thm:semmesexist}] 
Denote $d(x,y) = r$. By the $1$-Poincar\'e inequality 
and~\cite[Theorem~2]{K}, there exists a $C$ such that we have
\[
\Mod_{1,\nmu^C_{xy}}(\Gamma^C_{xy}) = \inf \int_X \rho \, d\nmu^C_{xy}>\frac{1}{C},
\]
where the infimum is over non-negative Borel functions $\rho$ with $\int_\gamma \rho \geq 1$ for every 
$\gamma \in \Gamma^C_{xy}$. Note that the estimates in~\cite{K} give 
the modulus bound for the set of all rectifiable curves between $x,y$, 
but the collection of curves that are longer than $4C^2d(x,y)$ has modulus less than 
$1/(2C)$, and can be excluded using the subadditivity of the modulus. 

Another way of stating this estimate is that if $f$ is a non-negative continuous function, and 
$\int_X f\, d\nmu^C_{xy} < \infty$, then for every $\epsilon>0$ there exists a $\gamma \in \Gamma^C_{xy}$ such that
\[
\int_\gamma f ~ds \leq (C+\epsilon) \int_X f \,d\nmu^C_{xy},
\]
for otherwise, $\frac{f}{(C+\epsilon)\int_X f \, d\nmu^C_{xy}}$ would be admissible with a too small a norm. In particular,
\begin{equation}\label{eq:optim1}
\inf_{(\gamma,I) \in \Gamma^C_{xy}} \int_\gamma f \,ds \leq C \int_X f \,d\nmu^C_{xy}.
\end{equation}
Since $f$ is continuous and $\Gamma^C_{xy}$ is a compact family, the above infimum is a minimum.
Parametrizing the curves $\gamma$ by length we also get
\begin{equation} \label{eq:modest}
\int_{(\gamma,I) \in \Gamma^C_{xy}} \int_I f \circ \gamma(t) \, dt \, d\beta(\gamma, I)\leq C \int_X f \,d\nmu^C_{xy},
\end{equation}
where $\beta$ is the Dirac measure on $\Gamma^C_{xy}$ based at any of the optimal 
choices $(\gamma,I)$ that 
achieves the infimum in~\eqref{eq:optim1}.

Let $K$ be the set of probability measures $\alpha$ on $\Gamma^C_{xy}$; thus $K$ is a compact and convex set of 
measures with respect to weak* convergence. Set
\[
G = \left\{\ f : X \to [0,1] \ \left|\  f \text{ is continuous } \right.\right\} \subset C(X).
\]
Here $C(X)$ is the set of all continuous functions equipped with the uniform topology and $G$ is a closed convex subset thereof.
Then define $F : G \times K \to \R$ by
\[
F(f,\alpha) = C \int_X f \,d\nmu^C_{xy}-\int_{\Gamma^C_{xy}} \int_I f(\gamma(t)) \,dt \,d\alpha(\gamma,I).
\]
Clearly $F$ is continuous in $\alpha$, since $\Phi_f((\gamma,I)) = \int_I f(\gamma(t))\, dt$ is continuous in $\gamma$. Also, 
$F(\cdot,\alpha)$ is convex for every $\alpha \in K$, and $F(f,\cdot)$ is affine and {\em a fortiori} 
concave for any $f \in G$. Thus, we can apply Proposition~\ref{thm:minmax} to obtain
\[
\sup_{\alpha \in K} \inf_{f \in G} F(f,\alpha) = \inf_{f \in G} \sup_{\alpha \in K} F(f,\alpha).
\]
Now, for $f\in G$, by estimate~\eqref{eq:optim1} we have $\sup_{\alpha \in K} F(f,\alpha) \geq 0$. 
Thus, we get
\[
\sup_{\alpha \in K} \inf_{f \in G} F(f,\alpha) \geq 0.
\]
In particular, for every $\epsilon>0$ and every $f\in G$ there exists a $\alpha_\epsilon \in K$, such that
\[
F(f,\alpha_\epsilon) \geq -\epsilon.
\]
Since for each $f\in G$ the map $K\ni\alpha\mapsto F(f,\alpha)$ is continuous, 
we can extract a weakly convergent sequence 
$\alpha_{\epsilon_i}\rightharpoonup \alpha_{xy} \in K$ (with $\epsilon_i \to 0$ as $i\to\infty$), such that for every $f \in G$
\[
F(f,\alpha_{xy}) \geq 0.
\]
Now, recalling the definition of $F$,  for every $f \in G$,
\[
\int_{\Gamma^C_{xy}} \int_I f(\gamma(t)) \, dt \, d\alpha_{xy}(\gamma,I) \leq C\int_X f \, d\nmu^C_{xy}.
\]
Also, since the curves $\gamma$ are $1$-Lipschitz, it follows that $\int_\gamma f ~ds \leq \int_I f(\gamma(t))\, dt$, and 
$\alpha_{xy}$ induces a measure (which we denote by the same symbol) on 
$\Gamma_{xy} = \{\gamma\, :\,  (\gamma,I) \in \Gamma^C_{xy}\text{ for some }I\}$. With respect to 
this measure, we have for every $f \in C(X)$ with $0 \leq f \leq 1$ that
\[
\int_{\Gamma_{xy}} \int_\gamma f \,ds ~d\alpha_{xy}(\gamma) \leq C\int_X f \,d\nmu^C_{xy}.
\]
 By a limiting argument we obtain the same inequality for $f=\chi_A$ corresponding to
 Borel sets $A\subset X$, and thus the measure $\sigma_{xy}=\alpha_{xy}$, which is supported on the compact set 
 $\Gamma_{xy}$, constitutes a Semmes family of curves in the sense of Definition \ref{def:semmes}, and the proof is complete. 
\end{proof}

Each Borel function in $L^1_{loc}(X)$ can be approximated by simple Borel functions. Hence
it follows from~\eqref{pencil} that 
\begin{equation}\label{pencil2}
\int_{\Gamma_{xy}}\int_\gamma g \, ds \, d\sigma(\gamma)\le C\int_{C B_{x,y}}g(z)R_{xy}(z)\, d\mu(z),
\end{equation}
for Borel functions $g:C B_{x,y}\to\R$. Doubling metric measure spaces supporting a Semmes pencil 
curves support a $1$-Poincar\'e inequality (see e.g. the discussion following \cite[Definition 14.2.4]{Se}).
In what follows we prove that they also support the $\AM$-Poincar\'e inequality.
Recall that
$$
I_A(u)(x)=\int_A\frac{u(z) d(x,z)}{\mu(B(x,d(x,z)))}\, d\mu(z)
$$
denotes the \em Riesz potential \em of a non-negative function $u$ defined on $X$ on a subset $A\subset X$.

\begin{prop}\label{SemmesAM}
If $X$ supports a Semmes pencil of curves, then $X$ supports the
$\AM$-Poincar\'e inequality.
\end{prop}

\begin{proof}
Let  $u\in L^1_{loc}(X)$ and let $(g_i)$ be a $\BV$-upper 
bound of $u$, and let $N$ be the collection of all points $x\in X$ for which
\[
\limsup_{r\to 0^+}\jint_{B(x,r)}|u-u(x)|\, d\mu>0;
\]
Then $\mu(N)=0$. We focus on points $x,y\in X\setminus N$. Then for each $\eps>0$ we know that the sets
\[
E_\eps(x):=\{z\in X\, :\, |u(z)-u(x)|>\eps\}, \qquad E_\eps(y)=\{z\in X\, :\, |u(z)-u(y)|>\eps\}
\]
satisfy
\[
\limsup_{r\to 0^+}\frac{\mu(B(x,r)\cap E_\eps(x))}{\mu(B(x,r))}=0=\limsup_{r\to 0^+}\frac{\mu(B(y,r)\cap E_\eps(y))}{\mu(B(y,r))}.
\]
We can inductively choose a strictly decreasing sequence $r_i>0$ such that $r_1<d(x,y)/4$, 
$r_{i+1}<r_i/4$, and
\[
\frac{\mu(B(x,r_i)\cap E_\eps(x))}{\mu(B(x,r_i))}<\frac{2^{-i}}{2C_d},\qquad 
\frac{\mu(B(y,r_i)\cap E_\eps(y))}{\mu(B(y,r_i))}<\frac{2^{-i}}{2C_d}.
\]
For each $i$ let $\Gamma_i(x)$ denote the collection of all $\gamma\in\Gamma_{xy}$ such that 
\[
\mathcal{H}^1(\gamma^{-1}([B(x,r_i)\setminus B(x,r_i/2)]\setminus E_\eps(x)))=0, 
\]
and $\Gamma_i(y)$ the analogous family with $y$ playing the role of $x$. By the fact that $\Gamma_{xy}$ is a Semmes
family and by the fact that $\mu$ is doubling, we have that 
\[
\frac{r_i}{2}\sigma_{xy}(\Gamma_i(x))
\le \int_{\Gamma_{xy}}\ell(\gamma\cap E_\eps(x)\cap B(x,r_i)\setminus B(x,r_i/2))\, d\sigma_{xy}(\gamma)
\le C_d\, \frac{r_i}{\mu(B(x,r_i))}\, \mu(E_\eps(x)\cap B(x,r_i)),
\]
and so by the choice of $r_i$ we have
\[
\sigma_{xy}(\Gamma_i(x))\le 2^{-i}.
\]
Hence for each positive integer $n$,
\[
\sigma_{xy}\left(\bigcup_{i=n}^\infty \Gamma_i(x)\right)\le 2^{1-n},
\]
and so with
\[
\Gamma(x)=\bigcap_{n\in\N}\bigcup_{i=n}^\infty \Gamma_i(x),
\]
we have that $\sigma_{xy}(\Gamma(x))=0$. Note that if $\gamma\in\Gamma_{xy}\setminus\Gamma(x)$, then
whenever $N_\gamma$ is a subset of the domain of $\gamma$ with $\mathcal{H}^1(N_\gamma)=0$, 
we can find a sequence of points $x_i\in\gamma\setminus [E_\eps(x)\cup\gamma(N_\gamma)]$ such that 
$x_i\to x$ as $i\to\infty$. Let $\Gamma(y)$ be the analogous subfamily of curves with respect to the point $y$;
then $\sigma_{xy}(\Gamma(x)\cup\Gamma(y))=0$. 
Let $(g_i)$ be a $\BV_{\AM}$-upper bound for $u$.
For $\gamma\in\Gamma_{xy}\setminus[\Gamma(x)\cup\Gamma(y)]$, we set $N_\gamma$ to be the set of points
in the domain of $\gamma$ that forms the exceptional set in the condition~\eqref{eq:AM-uppr}, and we select 
the sequences $x_i, y_i$ as above. Then
we have that
\[
|u(x)-u(y)|-2\eps\le \liminf_{i\to\infty} |u(x_i)-u(y_i)|\le \liminf_{i\to\infty}\int_\gamma g_i\, ds.
\]
Therefore, for $x,y\in X\setminus N$ and for each $\gamma\in \Gamma_{xy}\setminus(\Gamma(x)\cup\Gamma(y))$,
we have 
\[
|u(x)-u(y)|-2\eps\le \liminf_{i\to\infty}\int_\gamma g_i\, ds.
\]
We then have by Fatou's lemma
and~\eqref{pencil2} that
\begin{equation*}
\begin{split}
|u(x)-u(y)|-2\eps\le  & 
\int_{\Gamma_{xy}} \liminf_{i\to\infty}\int_\gamma g_i\, ds\,  d\sigma_{xy}(\gamma)\\
\leq &\liminf_{i\to\infty} \int_{\Gamma_{xy}} \int_\gamma g_i\, ds d\sigma_{xy}(\gamma)\\
\leq & C\liminf_{i\to\infty}\int_{C B_{x,y}}g_i(z)R_{xy}(z)\, d\mu(z)\\
 \le &\int_{C B_{x,y}}R_{xy}(z)\, d\nu(z)\\
\le & C(I_{C B_{x,y}}\nu(x)+I_{C B_{x,y}}\nu(y)),
\end{split} 
\end{equation*}
where $\nu$ is the Radon measure as in Remark~\ref{rem:truncation}. The constant $C$ in the above does
not depend on $\eps$; hence, by letting $\eps\to 0^+$ we obtain
\[
|u(x)-u(y)|\le C(I_{C B_{x,y}}\nu(x)+I_{C B_{x,y}}\nu(y))
\]
whenever $x,y\in X\setminus N$. 
For $x,y\in B$ with $R$ the radius of $B$, setting $B_i=B(x,2^{-i}Cd(x,y))$ for $i\in\N\cup\{0\}$, we see that
\begin{align*}
I_{CB_{x,y}}\nu(x)\le \int_{B(x,Cd(x,y))}\frac{d(x,z)}{\mu(B(x,d(z,x)))}\, d\mu(z)
  &\le C\, \sum_{i=0}^\infty \frac{2^{-i}Cd(x,y)}{\mu(B_i)}\, \nu(B_i)\\
  &\le C\, d(x,y)\, h_B(x)\, \sum_{i=0}^\infty 2^{-i},
\end{align*}
where
\[
h_B(x)=\sup_{0<r\le CR} \frac{\nu(B(x,r))}{\mu(B(x,r))}.
\]
Thus $h_B$ is a Haj\l asz gradient of $u$ in $B$, 
that is,
\[
|u(x)-u(y)|\le Cd(x,y)[h_B(x)+h_B(y)]
\]
for $\mu$-a.e.~$x,y\in B$,
and we have the weak inequality
\[
\mu(\{x\in B\, :\, h_B(x)>t\})\le C\, \frac{\nu(B)}{t} \text{ for }t>0.
\]
Thus $h_B\in L^q(B)$ for $0<q<1$, and hence $u\in M^{1,q}(B)$ in the sense of~\cite{Haj-Contem},
and so by~\cite[Corollary~8.9 of page~202]{Haj-Contem} or by~\cite[Proposition~2.4]{KLS}, we know that 
\[
\jint_B|u-u_B|\, d\mu\le C\, R\, \frac{\nu(B)}{\mu(B)}.
\]

%
%
The proof is then completed by taking a sequence of sequences 
$(g_i^j)_i$
 that are $\BV_{\AM}$-upper bound
of $u$ with corresponding measures $\nu_j$ such that $\lim_j\nu_j(2B)=\Vert D_{\AM}u\Vert(2B)$.
\end{proof}

From Proposition~\ref{SemmesAM}, Theorem~\ref{thm:semmesexist}, and
Proposition~\ref{prop:BV-AM-PI} we have the following.

\begin{thm}\label{thm:main1}
Let $\mu$ be a doubling measure on $X$. Then the following are equivalent:
\begin{enumerate}
\item $X$ supports a $1$-Poincar\'e inequality.
\item $X$ supports a Semmes pencil of curves.
\item $X$ supports an $\AM$-Poincar\'e inequality.
\end{enumerate}
In any (and therefore all) of the above, we have  $\BV(X)=\BV_{\AM}(X)$ and $N^{1,1}(X)=N^{1,1}_{\AM}(X)$.
\end{thm}

\end{document}